\documentclass[11pt]{article}
\usepackage{amsfonts}
\usepackage[width=17cm, left=1cm]{geometry}
\usepackage{amssymb}
\usepackage{amsmath}
\usepackage{amsthm}
\usepackage{color}
\usepackage{hyperref}
\usepackage{graphicx}

\newtheorem{proposition}{Proposition}

\newtheorem{theorem}{Theorem}
\newtheorem{lemma}{Lemma}
\newtheorem{corollary}{Corollary}
\theoremstyle{definition}
\newtheorem{definition}{Definition}
\newtheorem{remark}{\noindent Remark}
\newcommand{\eps}{\varepsilon}

\newcommand{\oorb}{\mathrm{ORB}}
\newcommand{\T}{\mathcal{T}}
\newcommand{\OT}{\mathcal{OT}}
\newcommand{\OTP}{\mathcal{OTP}}
\newcommand{\ot}{\mathrm{ot}}
\newcommand{\otp}{\mathrm{otp}}
\newcommand{\cotp}{\mathrm{cotp}}
\newcommand{\ii}{\mathfrak{I}}
\newcommand{\iz}{\mathcal{I}}
\newcommand{\uag}{UA}

\title{Combinatorial invariants of metric filtrations and automorphisms; the universal adic graph\thanks{Supported by the RSF grant 17-71-20153.}}

\author{A.~M.~Vershik\thanks{%
St.~Petersburg Department of Steklov Institute of Mathematics,
St.~Petersburg State University, and Institute for Information Transmission Problems.
E-mail: {\tt avershik@gmail.com}.} \and P.~B.~Zatitskiy\thanks{%
St.~Petersburg State University and St.~Petersburg Department of Steklov Institute of Mathematics.
E-mail: {\tt pavelz@pdmi.ras.ru}.} }

\date{}
\begin{document}

\maketitle

\tableofcontents

\abstract{We suggest a combinatorial classification of metric filtrations in measure spaces; a complete invariant of such a filtration is its combinatorial scheme, a measure on the space of hierarchies of the group~$\mathbb Z$. In turn, the notion of combinatorial scheme is a source of new metric invariants of automorphisms approximated via basic filtrations. We construct a universal graph endowed with an adic structure such that every automorphism can be realized in its path space.}

\section{Introduction}
\subsection{Metric filtrations and their applications}

This paper deals with applications of the theory of metric filtrations (see~\cite{V17} and the references therein) to uniform approximation of automorphisms of measure spaces and the analysis of adic transformations in path spaces of graphs. Conceptually, it is closely related to the first author's work on dyadic and homogeneous sequences of measurable partitions (= filtrations), standardness, the ``scale'' metric invariant, etc.\ (see
\cite{V68,V70,V71,V73f,V73d,V94,VG07}). Essentially, we pass from homogeneous filtrations to arbitrary ones, pose a number of problems on so-called {\it combinatorially definite} (standard) filtrations, and relate them to properties of automorphisms being approximated. On the other hand, the main example of filtrations is provided by so-called tail filtrations in path spaces of graded graphs, or, equivalently, spaces of realizations of Markov chains, so we arrive at a realization of automorphisms as adic transformations (see \cite{V81, V82}).

It is characteristic of numerous classification problems in ergodic theory that the most important objects (automorphisms, group actions) have no nontrivial finite invariants, i.e., metric invariants arising from finite approximations or finite projections. This means that classification problems are of purely asymptotic nature. An illustration of this point is, for example, the classical Rokhlin's lemma, which says that every aperiodic automorphism (more exactly, every free action of the group~$\mathbb Z$ with invariant measure) can be approximated with any accuracy in all reasonable metrics by periodic automorphisms. Hence, to obtain nontrivial invariants of automorphisms, one should  consider infinite sequences of periodic approximations and their invariants.

There are two theories of infinite approximations by periodic transformations: the theory of weak approximations, in the weak (operator) topology, successfully developed in the 1960s--1970s by A.~Katok, A.~Stepin, and others (see  \cite{KS,KH,St71}), and the theory of uniform approximations, in the uniform metric, initiated by the first author in the 1960s simultaneously with the theory of filtrations, i.e., decreasing sequences of $\sigma$-algebras or measurable partitions, see the references above. The theory of uniform approximations and orbit theory were the main applications of the theory of filtrations. Another important area of application for the theory of filtrations, which we do not touch upon in this paper, is the theory of stationary filtrations arising as decreasing sequences of $\sigma$-algebras of ``pasts'' of stationary random processes, or, which is the same, sequences of preimages of the full $\sigma$-algebra under powers of faithful endomorphisms.

In this paper, we explain that a monotone sequence of uniform approximations of an automorphism in a measure space determines,  in a natural way,  a filtration whose partitions are orbit partitions of periodic automorphisms. This filtration is special in the sense that it is endowed with an order and is semihomogeneous; in other words, it inherits two approximation structures: a~coherent ordering of points in almost all elements of all partitions (a linear order in the group~$\mathbb Z$) and semihomogeneity, i.e., the uniformness of almost all conditional measures, which follows from the invariance of the measure under the automorphism. In terms of the theory of graded graphs, this means that the graph is endowed with a structure of a linear order on the edges entering each vertex (an ``adic structure''), and that the measure on the path space is central, i.e., the conditional measure on initial segments of paths is uniform. Moreover, a semihomogeneous filtration endowed with such an order uniquely determines the corresponding automorphism (without order, one cannot recover the automorphism from the filtration up to isomorphism).  From this viewpoint, the filtration approach and uniform approximation are related to the problem of metric isomorphism more closely than weak approximation. The study and construction of invariants of automorphisms and groups of automorphisms is preceded by the study of invariants of filtrations.

 {\it All metric invariants of filtrations fall into two classes: combinatorial (finite) invariants and transfinite ones}. Combinatorial invariants are invariants of all finite fragments of filtrations, i.e., invariants of periodic approximations; they are described below and represent some measures on the space of hierarchies on the group~$\mathbb Z$. The prospect of obtaining an efficient combinatorial classification of filtrations described below was observed in~\cite{V17}, but the fact that this classification problem has indeed turned out to be tame gives hope for further classifications.

Transfinite invariants, whose existence is not obvious, are not combinatorial; their study requires considering deeper properties of filtrations, related to the notion of standardness or combinatorial definiteness. Standard, or combinatorially definite, filtrations are filtrations that are uniquely determined up to metric isomorphism by the combinatorial invariants. For example, a dyadic standard filtration is a Bernoulli filtration, it is combinatorially definite and can be recovered from its one-dimensional distribution.

By the lacunary theorem (see~\cite{V68,V17} and Section~4), every filtration contains a ``thinning''  that is already a combinatorially definite filtration; thus, every automorphism becomes combinatorially definite with respect to some ``thinning.'' So, we obtain a chain of invariants also for filtrations that are not combinatorially definite, and for general automorphisms.

The adic dynamics (see \cite{V81,V82}), i.e., a special transformation of the path space of a graded graph, has already given many new nontrivial examples of dynamical systems. For instance, the Pascal automorphism (whose spectrum is still unknown) is combinatorially definite in the sense of this paper. In general, adic transformations are of great interest both from theoretical and practical point of view. Here we suggest constructions of universal graphs on which every automorphism can be realized as an adic shift. The study of specific automorphisms has already begun (see the survey~\cite{BK15}, and also~\cite{GJ02}), but, according to the conclusion of the paper~\cite{V17} (see also below), the problem of classification of combinatorially definite filtrations is tame, i.e., there is a manageable space of classes, or orbits, or complete metric invariants of such filtrations (see the definition of combinatorial schemes of hierarchies below). This gives a chance to obtain a reasonable combinatorial classification of measure-preserving automorphisms. One may hope that such an approach is also possible for actions of other (amenable) groups, primarily for locally finite groups, such as~$\sum {\mathbb Z}_p$, and for the lattices~${\mathbb Z}^d$.

\subsection{Rokhlin's lemma, a sequence of uniform approximations, basic filtrations}

Recall,  in a form suitable for our purposes,  the well-known Rokhlin's lemma on approximation of a measure-preserving aperiodic automorphism by periodic automorphisms in the uniform metric. The uniform metric on the space of transformations of a space~$X$ preserving a measure~$\mu$ is defined as follows:
given transformations~$T$ and~$\tilde T$,
$$\rho(T,\tilde T)\equiv \mu\{x\in X: Tx \ne \tilde Tx\}.$$

\begin{lemma}[Rokhlin's lemma]
For every $\varepsilon >0$, every positive integer $n \in \mathbb N$, and every measure-preserving aperiodic automorphism~$T$ in a Lebesgue space $(X,\mu)$ there exists a periodic automorphism~$T_n$ with period~$n$ such that
$$ \rho(T,T_n)<\frac{1}{n}+\varepsilon.$$
\end{lemma}

In what follows, it is useful to drop the condition that~$T_n$ has period~$n$ almost everywhere and assume that the periods $p_n(x)$ of $T_n$ can be different at different points $x \in X$ but uniformly bounded on a subset of full measure. This weaker assertion is even slightly easier to prove than the classical Rokhlin's lemma. Denote by $\oorb(T_n)$ the measurable partition of the space $(X,\mu)$ into the orbits of~$T_n$; we may assume that almost all elements of this partition are finite sets of the form $\{x,Tx,\dots, T^kx\}$ endowed with the uniform conditional measure and a linear order such that the resulting ordered space is isomorphic to an arbitrary interval of the line
(we set $y<Ty$ for every~$y$).

Greater freedom in choosing approximations is needed to construct a coherent sequence of approximations satisfying the following properties: $\rho(T_n,T)\to 0$, and the partitions $\oorb(T_n)$ become coarser, i.e., almost every orbit of~$T_{n+1}$ consists of several orbits of~$T_n$. Besides, for every~$n$ the quotient of $\oorb(T_{n+1})$ by $\oorb(T_n)$ is endowed with a linear order. Let us introduce an abstract definition for the resulting structure.

\begin{definition}
A filtration in a Lebesgue space $(X,\mu)$ with continuous measure is a decreasing sequence of measurable partitions $\Xi = \{\xi_n\}_{n\geq 0}$ with $\xi_0=\varepsilon$ being the partition into singletons. A~filtration is said to be ergodic if the measurable intersection $\bigwedge_n\xi_n$ is the trivial partition, which is usually denoted by~$\nu$.
  \end{definition}

We introduce the following special properties of filtrations. A filtration is said to
\begin{itemize}
\item be \textbf{locally finite} if for every $n$ almost all elements of~$\xi_n$ are finite sets, and the number of different types of conditional measures on elements of~$\xi_n$ is finite (depending on~$n$);
\item be \textbf{semihomogeneous} if the conditional measures on the elements of  $\xi_n$ are uniform for all~$n$;
\item \textbf{induce an order} if every element of the quotient partition $\xi_{n+1}/\xi_n$ is endowed with a~measurable linear order (measurability means that the set of all points with a given number in the elements of $\xi_{n+1}/\xi_n$ is measurable); these orders induce a coherent order on the elements of the partitions $\xi_n$, and hence on the classes of the limiting partition $\bigcap_n \xi_n$ (which is in general not measurable); we assume that the order type is $\mathbb{Z}$ for almost all classes.
\end{itemize}

Filtrations satisfying all these properties are said to be \emph{basic}.

\begin{definition}\label{def1}
Let $T$ be an aperiodic automorphism of a Lebesgue space $(X,\mu)$. We say that a~basic filtration $\Xi = \{\xi_n\}$ of $(X,\mu)$ is \emph{basic for $T$} if the corresponding order is induced by~$T$ and the limit of  $\xi_n$ is the partition into the orbits of~$T$~mod~0.
\end{definition}

Of course, for every aperiodic automorphism~$T$ there are many basic filtrations. If $T$ is ergodic, then its basic filtrations are also ergodic. Given a basic filtration for~$T$, one can construct a coherent sequence of automorphisms~$T_n$ approximating~$T$ in the sense described above. So, the language of basic filtrations and that of coherent sequences of approximating automorhisms are equivalent.

\section{Invariants of the combinatorial equivalence of filtrations and combinatorial definiteness of automorphisms}

\subsection{Combinatorial definiteness of automorphisms}

Let $A$ be an arbitrary finite set and $\{\eta_i\}_{i=0}^{n}$ be an ordered finite filtration on this set, the last partition $\eta_n$ being trivial (consisting of a single nonempty class). We construct an ordered graded tree corresponding to this finite filtration as follows. The vertices of level~$i$ in this tree correspond to the elements of~$\eta_i$. A vertex of level~$i+1$ is joined by an edge with a vertex of level~$i$ if the corresponding elements of partitions are nested. The $n$th level contains a single vertex, while the vertices of level~$0$ are the elements of~$A$. The set~$A$ is endowed with a linear order: the order from the definition of a filtration determines an order on the edges joining every vertex with vertices of the previous level. The obtained graded tree will be called the  \textit{filtration tree on the set~$A$} (see Fig.~\ref{fig1}). The set of all ordered graded finite trees will be denoted by~$\OT$. Besides, we will consider trees with a marked vertex (leaf). The set of all ordered graded finite trees with a marked leaf will be denoted by~$\OTP$.

\begin{figure}[h!]%
\begin{center}
\includegraphics[width=0.45\textwidth]{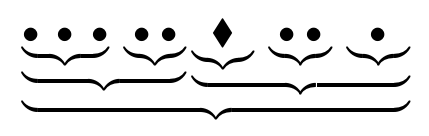}%
\includegraphics[width=0.45\textwidth]{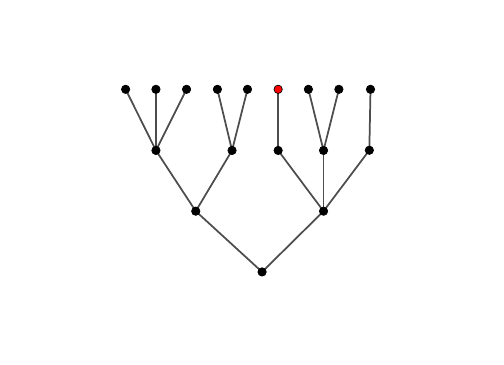}%
\end{center}
\caption{A finite filtration and its tree with a marked vertex}%
\label{fig1}%
\end{figure}

Let $\Xi = \{\xi_n\}_{n \geq 0}$ be a basic filtration on a space $(X,\mu)$. For $n \geq 0$ and $x \in X$, consider the ordered graded tree $\otp_n(x) \in \OTP$ corresponding to the restriction of the finite filtration $\{\xi_i\}_{i=0}^n$ on the element of $\xi_n$ containing $x$ with the marked leaf corresponding to~$x$. By $\ot_n(x)$ we denote the same ordered tree without marked vertex. Consider the partition $\bar\xi_n$ of the space $(X,\mu)$ into the preimages of points under the map~$\otp_n$. We say that the sequence $\bar \Xi$ of refining partitions $\{\bar\xi_n\}_{n \geq 0}$ is \textit{associated} with the basic filtration~$\Xi$.

\begin{definition}[see \cite{V70,V17}]
We say that a basic filtration $\Xi$ on a space $(X,\mu)$ is \emph{combinatorially definite} if the sequence of partitions $\bar\xi_n$, $n \geq 0$, is a basis in the space $(X,\mu)$, that is, it converges to the partition into singletons~mod~0 (in more detail, this means that there exists a subset $\tilde X \subset X$ of full measure such that for any two points $x,y \in \tilde X$ there exists $n$ such that $x$ and $y$ lie in different elements of the partition~$\bar\xi_n)$.
\end{definition}

Essentially, this definition singles out a class of basic filtrations that are completely determined, up to metric isomorphism, by the collection of invariants of their finite fragments. These invariants are of combinatorial nature, which explains the name.  The definition is inspired by another one (see~\cite{V17}), which singles out a class of arbitrary (not necessarily basic) filtrations called {\it standard}. The latter definition is, in turn, a generalization of earlier work on standard dyadic and homogeneous filtrations~\cite{V70,V94}. The important question of the relation between the notions of combinatorial definiteness and standardness will be considered later.

\subsection{Combinatorial equivalence of filtrations and the canonical quotient}

\begin{definition}[see \cite{V17}]
We say that two basic filtrations $\Xi^1 = \{\xi_n^1\}_{n\geq 0}$ and $\Xi^2= \{\tilde\xi_n^2\}_{n\geq 0}$ on Lebesgue spaces $(X^1,\mu^1)$ and  $(X^2,\mu^2)$, respectively, are \emph{combinatorially equivalent}, or \emph{have the same combinatorial type}, if for every $n \geq 0$ the finite filtrations $\{\xi_k^1\}_{k=0}^n$ of  $(X^1,\mu^1)$ and $\{\xi_k^2\}_{k=0}^n$ of  $(X^2,\mu^2)$ are metrically isomorphic.
\end{definition}

It is clear from the previous definition that for combinatorially definite basic filtrations, combinatorial equivalence coincides with metric isomorphism.

Let $\Xi$ be a basic filtration on a space $(X,\mu)$. Consider the equivalence relation on~$X$ determined by the associated sequence of measurable partitions $\bar \Xi$: points $x$ and $y$ lie in the same equivalence class if and only if $\otp_n(x)=\otp_n(y)$ for all $n \geq 0$. This equivalence relation is measurable and respects the order, hence we can take the  corresponding quotient. The resulting filtration will be called the
\textit{canonical quotient} of the basic filtration~$\Xi$.

\begin{remark}
The canonical quotient of a combinatorially definite basic filtration~$\Xi$ coincides with~$\Xi$ itself. The canonical quotient of an arbitrary basic filtration~$\Xi$ is combinatorially equivalent to~$\Xi$, but, in general, not metrically isomorphic to~$\Xi$.
\end{remark}

\begin{remark}
If $\Xi$ is a basic filtration for an automorphism~$T$, then the canonical quotient of~$\Xi$ determines a quotient of~$T$. However, unlike the quotient filtration, this quotient automorphism is not canonical, since it depends on the choice of an approximation.
\end{remark}

The canonical projection turns the space of all basic filtrations into a bundle over the space of all combinatorially definite basic filtrations. A very important and interesting question is whether the fibers of this bundle are isomorphic in some sense, i.e., whether the bundle is isomorphic to a~direct product.

\subsection{Random hierarchies on $\mathbb{Z}$ and classification of combinatorially definite filtrations}

\begin{definition}
A \emph{hierarchy on $\mathbb{Z}$} is a filtration on the space $\mathbb{Z}$ endowed with the counting measure that is basic for the left shift in the sense of Definition~\ref{def1}. By $\iz$ we denote the set of all hierarchies on $\mathbb{Z}$; this is a compact space in the natural topology.
\end{definition}

In other words, a \textit{hierarchy on $\mathbb{Z}$} is a coarsening sequence of partitions of $\mathbb{Z}$ such that every element of every partition is a finite interval of consecutive integers and any two integers lie in the same element of some partition with a sufficiently large number.

Let $\Xi$ be a basic filtration on a space $(X,\mu)$. With almost every point $x\in X$ we associate the hierarchy $\ii(x)$ on $\mathbb Z$ determined by the restriction of~$\Xi$ to the orbit of~$x$ identified with $\mathbb Z$ in a natural way (the point $T^n x$, $n \in \mathbb{Z},$ is identified with~$n$). Note that the hierarchy $\ii(x)$ can be understood as the inductive limit as $n \to \infty$ of the ordered graded trees $\otp_n(x)$ with marked vertices.

The image of the measure $\mu$ under the map $\ii$ is a shift-invariant measure on the space $\iz$. Thus, a basic filtration $\Xi$ determines an \emph{invariant random hierarchy} on $\mathbb{Z}$. This measure will be called the \emph{combinatorial scheme} of~$\Xi$. A basic filtration $\Xi$ is combinatorially definite if and only if the map~$\ii$ is injective mod~0. Given an invariant measure~$\nu$ on the space $\iz$, there exists a space $(X,\mu)$ and a basic combinatorially definite filtration~$\Xi$ on this space such that $\nu$ is the combinatorial scheme of~$\Xi$.

\begin{theorem}
In the class of combinatorially definite filtrations, the combinatorial scheme is a~complete metric invariant.
\end{theorem}

We illustrate the introduced notion with a simplest example. A measure on the space~$\iz$ of hierarchies is determined by its values on the cylinders, i.e., sets of hierarchies having a fixed structure on a given element of the partition of level~$n$ containing~$0$. In the case where $\Xi$ is a~dyadic filtration on a Lebesgue space $(X,\mu)$, the hierarchies $\ii(x)$ corresponding to points~$x$ of~$X$ are also dyadic. The only shift-invariant measure concentrated on dyadic hierarchies is uniform: all cylinders determined by elements of the partition of level~$n$ have equal probabilities.

\begin{definition}
We say that an automorphism~$T$ of a space $(X,\mu)$ is \emph{combinatorially definite} with respect to a combinatorial scheme if there is a combinatorially definite basic filtration of~$T$ with this combinatorial scheme.
\end{definition}

\begin{theorem}
Every automorphism is combinatorially definite with respect to some combinatorial scheme.
 \end{theorem}

The proof of this theorem essentially follows from an analog of the lacunary theorem, see Section~4, Corollary~\ref{cor1}.

\emph{The collection of combinatorial schemes with respect to which an automorphism~$T$ is combinatorially definite will be called the combinatorial scheme of~$T$. Some characteristics of this scheme are metric invariants of the automorphism.}

It is of interest to study the behavior of the combinatorial scheme of an automorphism with respect to various operations (taking a derivative or integral automorphism, the product of automorphisms, etc.).

One can easily see a similarity between this definition and that of the scale of an automorphism (see~\cite{V73f}), which is exactly one of the invariant characteristics of the combinatorial scheme and the automorphism itself. In more detail this relation will be discussed elsewhere.

The most interesting class consists of automorphisms with the simplest possible combinatorial scheme, the dyadic one; it is this scheme that is related to the notion of a measure-preserving automorphism with complete scale
(see~\cite{V73f,V17}). In a later paper \cite{Kat}, the notion of a standard automorphism was defined (the term is chosen by analogy with the notion of a standard dyadic filtration introduced in~\cite{V70}); the definition involves the notion of monotone equivalence in the sense of Kakutani. Apparently, the standardness of an automorphism in the sense of~\cite{Kat} is close to our combinatorial definiteness of an automorhism with respect to the dyadic scale.

The concept of the combinatorial scheme of an automorphism also covers substitutional ergodic theorems
related to the scale, see~\cite{V73f}.

\section{Realization of colored filtrations as tail filtrations on path spaces of graphs}

As already mentioned, adic shifts on path spaces of graded graphs are important examples of automorphisms of measure spaces. On the other hand, in \cite{V81,V82} the first author proved that every ergodic automorphism has an adic realization.

\begin{theorem}[\cite{V82}]\label{thAR}
For every ergodic automorphism~$T$ of a Lebesgue space there is a graded graph~$\Gamma$ endowed with an adic structure and a central measure $\mu$ on the path space~$\T(\Gamma)$ of~$\Gamma$ such that the adic shift $(\T(\Gamma),\mu)$ is isomorphic to~$T$.
\end{theorem}

The language of graded graphs is closely related to the language of basic filtrations. Let $\Gamma$ be a~graded graph endowed with an adic structure; the space  $\T(\Gamma)$ of infinite paths in~$\Gamma$ is equipped with the tail filtration $\Xi=\{\xi_n\}$ determined by the structure of the graph: two paths lie in the same element of the partition $\xi_n$ if they coincide starting from the $n$th level. Let $\mu$ be a central measure on $\T(\Gamma)$ such that almost every path has a successor and a predecessor in the sense of the adic order; in what follows, such a measure is said to be an \emph{essential central} measure. Then $\Xi$ is a~basic filtration on the space $(\T(\Gamma),\mu)$, the corresponding order being determined by the adic structure.

Let $\Gamma$ be a graded graph endowed with an adic structure. With each vertex~$v$ of~$\Gamma$ we associate an ordered graded tree~$\ot(v) \in \OT$ according to the following rule. With the vertex of level~$0$ we associate the tree consisting of a single vertex. Then we apply the following recursive (on  $n$, where $n \geq 0$) procedure. Let $w$ be a vertex of level~$n+1$ in~$\Gamma$, and let $(v_1,w), \dots, (v_k,w)$ be all edges leading to~$w$ from vertices of level~$n$ in the adic order (some vertices $v_i$ may be repeated). The graded tree~$\ot(w)$ is defined as follows: its root corresponds to the vertex~$w$ itself; there are $k$ edges joining it with the vertices corresponding to $v_1,\dots,v_k$ (counting multiplicities), the  ordered graded tree $\ot(v_i)$, $i=1,\dots, k$, already defined at the previous step hanging from each of these vertices as a root. An order on the edges leading from the root is determined by the adic structure of the graph~$\Gamma$.

The leaves of the constructed tree~$\ot(w)$ are in a one-to-one correspondence with the paths in~$\Gamma$ leading to~$w$ from the vertex of level~$0$, and the order on them corresponds to the adic order. The graded tree~$\ot(w)$ in which the leaf corresponding to a path~$x$  to~$w$ is marked will be denoted by~$\ot(x)$.

Thus we have defined a map~$\ot$ from the set of vertices of the graph~$\Gamma$ to the set $\OT$ of ordered graded trees; we have also defined a map, denoted by the same symbol, that sends finite paths starting at the vertex of level~$0$ in~$\Gamma$  to ordered graded trees with marked vertices.

\begin{definition}\label{mingraph}
Let $\Gamma$ be a graded graph endowed with an adic structure. We say that $\Gamma$ is \emph{minimal} if for any two vertices~$v,w$ of the same level, the trees~$\ot(v)$ and~$\ot(w)$ are different.

Let $\mu$ be an essential central measure on the space~$\T(\Gamma)$ of infinite paths in~$\Gamma$. We say that the measure $\mu$ is \emph{minimal} if almost all infinite paths
differ in the graded trees corresponding to their initial segments (i.e., there exists a subset  of full measure in~$\T(\Gamma)$ such that for any two paths $x$ and $y$  from this subset there is $n\geq 1$ such that $\ot(x[n])\ne \ot(y[n])$, where $x[n]$ and $y[n]$ are the initial segments  of length~$n$ of $x$ and $y$, respectively).
\end{definition}

\begin{proposition}\label{prop1}
A basic filtration $\Xi$ on a space $(X,\mu)$ is combinatorially definite if and only if it is isomorphic to the tail filtration of a minimal graded graph endowed with an adic structure and an essential central measure on the path space.
\end{proposition}

\begin{proof}
Let $\Xi$ be a combinatorially definite basic filtration. We construct a graded graph in which the vertices of level~$n$, $n \geq 0$, correspond to the different types of ordered trees $\ot_n(x)$, $x \in X$. Two vertices of neighboring levels are joined by an edge if they correspond to nested ordered trees; an order on the edges entering each vertex is determined in a natural way by the order of the corresponding tree. With each point $x\in X$ we associate the infinite path in the constructed graph that passes through the vertices corresponding to the trees $\ot_n(x)$, $n \geq 0$, and edges corresponding to the embeddings of $\otp_n(x)$ into $\otp_{n+1}(x)$. Since $\Xi$ is combinatorially definite, the resulting map is injective mod~0, and it sends the order of~$\Xi$ to the adic order on the graph. It follows that the pullback of~$\mu$ under this embedding is a central measure. It is an essential central measure, since the filtration is basic. Obviously, the tail filtration of the constructed graph endowed with this measure and the adic order is isomorphic to the original basic filtration~$\Xi$. The minimality of the graph follows from the construction: one can easily check that for every vertex~$v$ the ordered graded tree~$\ot(v)$ is exactly the tree with which $v$ is associated in the construction.

Conversely, assume that we have a minimal graded graph endowed with an adic structure and $\Xi$ is its tail filtration. If $x$ is a path in this graph, then $\ot_n(x) = \ot(x[n])$, where $x[n]$ is the initial segment of $x$ of length $n$. Since the graph is minimal, every path is uniquely determined by the collection of graded trees $\otp_n(x)$, $n \geq 0$, which means exactly that $\Xi$ is combinatorially definite.
\end{proof}

\begin{remark}
Let $\Gamma$ be a graded graph and $\Xi$ be the tail filtration of~$\Gamma$. If $\mu$ is  an essential minimal central measure on the path space $\T(\Gamma)$, then $\Xi$ is a combinatorially definite basic filtration of the space~$(\T(\Gamma),\mu)$.
\end{remark}

Proposition~\ref{prop1} shows that only a combinatorially definite basic filtration can have an adic realization on a minimal graph. In the case of a basic filtration that is not combinatorially definite, the construction described in the proof of Proposition~\ref{prop1} determines a quotient of this filtration isomorphic to the tail filtration of a minimal graph; essentially, this is exactly the canonical quotient. To construct an adic realization of a basic filtration that is not combinatorially definite, it is convenient to use the language of colored filtrations.

\begin{definition}
Let $\xi$ be a measurable partition of a space $(X,\mu)$. We say that $\xi$ is a \emph{colored partition} if the quotient space $X/\xi$ is endowed with a finite measurable partition $c[\xi]$, called a \emph{coloring}, which assigns  \emph{colors} to the elements of $\xi$.

A basic filtration $\Xi = \{\xi_n\}_{n \geq 0}$ is called a \emph{colored filtration} if each partition $\xi_n$ is endowed with a coloring $c[\xi_n]$.
\end{definition}

Let $\Xi$ be a colored basic filtration of a space $(X,\mu)$. For every $n \geq 0$ and almost every point $x \in X$, the ordered trees~$\otp_n(x)$ considered above also become colored: the color of a vertex of level~$i$ is defined as the color of the corresponding element of the partition~$\xi_i$. Let $\cotp_n(x)$ be the colored tree $\otp_n(x)$ with a marked vertex. The measurable partition of the space $(X,\mu)$ into the preimages of points under the map~$\cotp_n$ will be denoted by~$\overline{c(\xi_n)}$.

\begin{definition}
We say that a colored basic filtration~$\Xi$ of a space $(X,\mu)$ is \emph{combinatorially definite} if the sequence of partitions $\overline{c(\xi_n)}$, $n \geq 0$, is a basis of~$(X,\mu)$.
\end{definition}

Let $\Gamma$ be a graded graph endowed with an adic structure and  $\Xi$ be the tail filtration on the path space $\T(\Gamma)$. Let $\mu$ be an essential central measure on $\T(\Gamma)$, not necessarily minimal. Then $\Xi$ is a basic filtration on the space $(\T(\Gamma),\mu)$. A natural coloring of the elements of the partition $\xi_n$, $n \geq 0$, is determined by the vertices of level~$n$ in~$\Gamma$: assign a color to each such vertex, and define the color of an element of~$\xi_n$ as the color of the vertex of level~$n$ lying on the paths from this element. The colored basic filtration thus defined will be called the \textit{canonical colored filtration of the graph~$\Gamma$}.

\begin{proposition}\label{rem2}
Let $\Gamma$ be a graded graph endowed with an adic structure and an essential central measure. Then its canonical colored filtration is combinatorially definite. Conversely, if a Lebesgue space $(X,\mu)$ is endowed with a combinatorially definite colored basic filtration~$\Xi$, then there exists a graded graph~$\Gamma$ and an essential central measure~$\tilde \mu$ on its path space $\T(\Gamma)$ such that the canonical colored filtration of~$\Gamma$ is isomorphic to~$\Xi$.
\end{proposition}
\begin{proof}
The proof reproduces the proof of Proposition~\ref{prop1}.
\end{proof}

\begin{proposition}\label{prop2}
For every basic filtration on a Lebesgue space there is a coloring such that the resulting colored basic filtration is combinatorially definite.
\end{proposition}

\begin{proof}
Let $\Xi = \{\xi_n\}_{n \geq 0}$ be a given basic filtration on a Lebesgue space $(X,\mu)$. Fix a sequence $\{\zeta_n\}_{n \geq 0}$ of finite partitions of the space $(X,\mu)$ that separates points mod~0. Since $\Xi$ is a basic filtration, for every $n \geq 0$ every element of the partition $\xi_n$ is a finite ordered set of size bounded by a constant  $a_n$ depending only on~$n$. Let us regard $\zeta_n$ as a coloring of the points of~$X$ into a~finite number, say $b_n$, colors. Then every element of~$\xi_n$ is a finite ordered colored set; in total, there are at most $\sum_{k=1}^{a_n} b_n^k$ possible colorings. Thus we have defined a coloring of the partition $\xi_n$. The filtration colored in this way is combinatorially definite. Indeed, for almost any two points $x,y \in X$ there is $n$ such that $x,y$ lie in different elements of the partition~$\zeta_n$. But then the colored trees $\otp_n(x)$ and $\otp_n(y)$ are different.
\end{proof}

Propositions~\ref{rem2} and~\ref{prop2} essentially describe the proof of Theorem~\ref{thAR}. The question about a realization of a periodic ergodic automorphism is meaningless, hence we may assume that the automorphism is aperiodic. In this case, we can construct a basic filtration, which is not necessarily combinatorially definite, but, according to Proposition~\ref{prop2}, can be colored in such a way as to become combinatorially definite. Proposition~\ref{rem2} gives a construction of an adic realization of a colored filtration, completing the proof of Theorem~\ref{thAR}.

\section{The universal adic graph}

In this section, we prove a strengthening of Theorem~\ref{thAR}. Namely, we prove that all (aperiodic) automorphisms of a Lebesgue space can be realized on a single special graph endowed with an adic structure; it suffices to vary only a central measure~$\mu$ on its path space. Earlier, a similar result was obtained for the class of automorphisms having the dyadic odometer as a quotient: all such automorphisms can be realized on the so-called graph of ordered pairs (for details, see~\cite{VZ17}).

\subsection{Construction of the uniadic graph}
Consider the following graded graph. Level~$0$ contains a single vertex. Having a set  $V_n$ of vertices of level~$n$, we define a set $V_{n+1}$ of vertices of level~$n+1$ as $V_{n+1} = V_n^2 \sqcup \mathrm{copy}(V_n)$. Every vertex $w\in V_n^2$ is understood as an ordered pair~$(u,v)$ of vertices of level~$n$, and we draw edges from~$u$ and~$v$ to~$w$, endowing them with a natural order: the edge~$(v,w)$ is greater than~$(u,w)$. Every vertex $w\in\mathrm{copy}(V_n)$ is understood as a copy of a vertex~$u$ of level~$n$, and we draw a unique edge from~$u$ to~$w$. The resulting graph endowed with an adic structure will be called the \emph{uniadic} graph and denoted by~$\uag$ (see Fig.~\ref{figuag}).
\begin{figure}[h!]%
\begin{center}
\includegraphics[width=0.4\columnwidth]{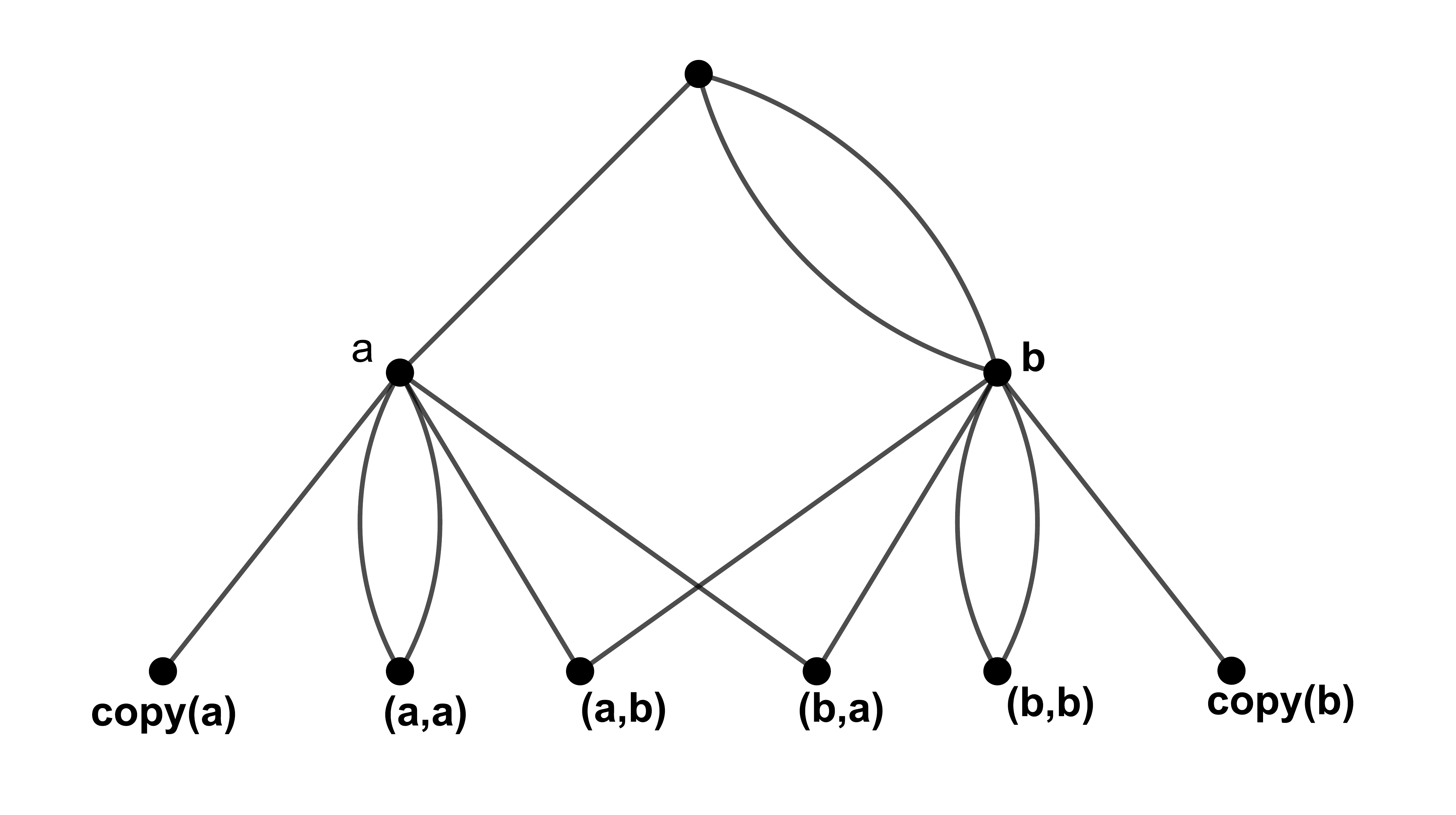}%
\end{center}
\vspace{-20pt}
\caption{Several first levels of the uniadic graph~\uag}%
\label{figuag}%
\end{figure}

Recall the definition of a telescoping of a graded graph.

\begin{definition}
Let $\Gamma$ be a graded graph and $\{k_n\}_{n\geq 0}$ be a strictly increasing sequence of nonnegative integers with $k_0=0$. We define a telescoping of~$\Gamma$ as follows. The vertices of level~$n$ in the new graph correspond to the vertices of level~$k_n$ in~$\Gamma$, and two vertices of neighboring levels in the new graph are joined by an edge of multiplicity equal to the number of paths in~$\Gamma$ between the corresponding vertices. An adic order on the edges of the new graph is determined by the adic order on the corresponding paths in the original graph.
\end{definition}

\begin{definition}
We say that a graded graph $\Gamma_1$ is an induced subgraph of a graded graph $\Gamma$ if the set of vertices and the set of edges of~$\Gamma_1$ are subsets of the set of vertices and and the set of edges of~$\Gamma$, respectively, and, besides, if $v$ is a vertex of~$\Gamma_1$, then $\Gamma_1$ contains all edges of~$\Gamma$ coming to~$v$ from vertices of the previous level. An order on the edges is inherited in a natural way.
\end{definition}

The following properties are clear from definitions.

\begin{remark}
If a graded graph $\Gamma_1$ is an induced subgraph of a graded graph~$\Gamma$, then the space of infinite paths in $\Gamma_1$ is a subset of the space of infinite paths in~$\Gamma$ invariant under the adic shift on~$\Gamma$.
\end{remark}

\begin{remark}
If a graded graph $\Gamma_1$ is a telescoping of a graded graph $\Gamma$, then the adic shifts on these graphs are isomorphic.
\end{remark}

\subsection{The universality theorem}
The uniadic graph $\uag$  is universal in the following sense.

\begin{proposition}\label{prop3} Let $\Gamma$ be a graded graph endowed with an adic structure such that every vertex has at least two edges entering it from above. Then there exists an induced subgraph~$\Gamma_1$ of $\uag$ such that some telescoping of~$\Gamma_1$ is isomorphic to $\Gamma$ (with the isomorphism respecting the adic order).\footnote{Proposition~\ref{prop3} is, in a sense, akin to the lacunary theorem (see~\cite{V68,V17}): an appropriately thinned filtration becomes combinatorially definite (in the lacunary theorem, standard).}
\end{proposition}

\begin{proof}
To prove this, it suffices to realize that for every $n\geq 0$, the bipartite graph formed by the $n$th and $(n+1)$th level of~$\Gamma$ can be extended by several intermediate levels so that the resulting graded graph (with finitely many levels) satisfies the following two properties: every vertex of every level (except the topmost one) has either one edge  or an ordered pair of edges entering it from above,
 and no two vertices have the same ancestors (taking into account the order); a telescoping of this graph coincides with the original bipartite graph.

The existence of such a thinning of the bipartite graph can be proved as follows. First, by an appropriate thinning, we ensure that no level contains two vertices with the same ancestors. The rest can be easily proved by induction on the number of edges in the bipartite graph. If some vertex~$w$ of the bottom level has more than two incoming edges, say from vertices $v_1,\dots,v_k$ (taking into account the order), then insert an intermediate level consisting of a copy of the top level with one additional vertex~$(v_1,v_2)$. Join this vertex by two upward edges with $v_1$ and $v_2$, exactly in this order. Join all vertices of the bottom level except~$w$ with copies of vertices of the top level as in the original bipartite graph. Finally, join $w$ with the vertex $(v_1,v_2)$ and the copies of the other vertices $v_i$, $i=3,\dots,k$ (see Fig.~\ref{fig2}). The path space of the new graph is isomorphic to the path space of the original graph, but now it remains to construct a thinning of a bipartite graph with one edge less.

\begin{figure}[h!]
\begin{center}
\includegraphics[width=0.45\columnwidth]{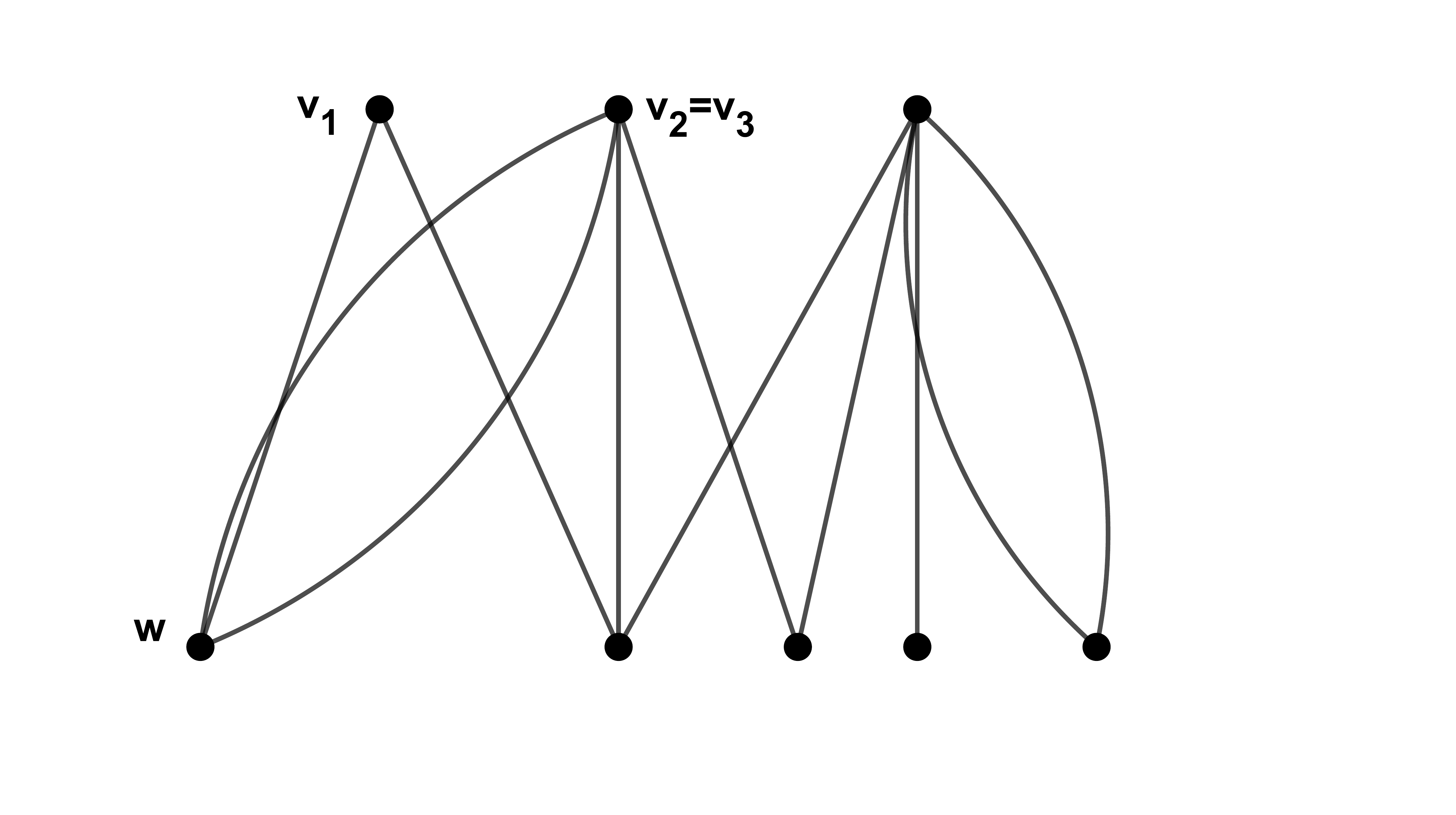}%
\includegraphics[width=0.45\columnwidth]{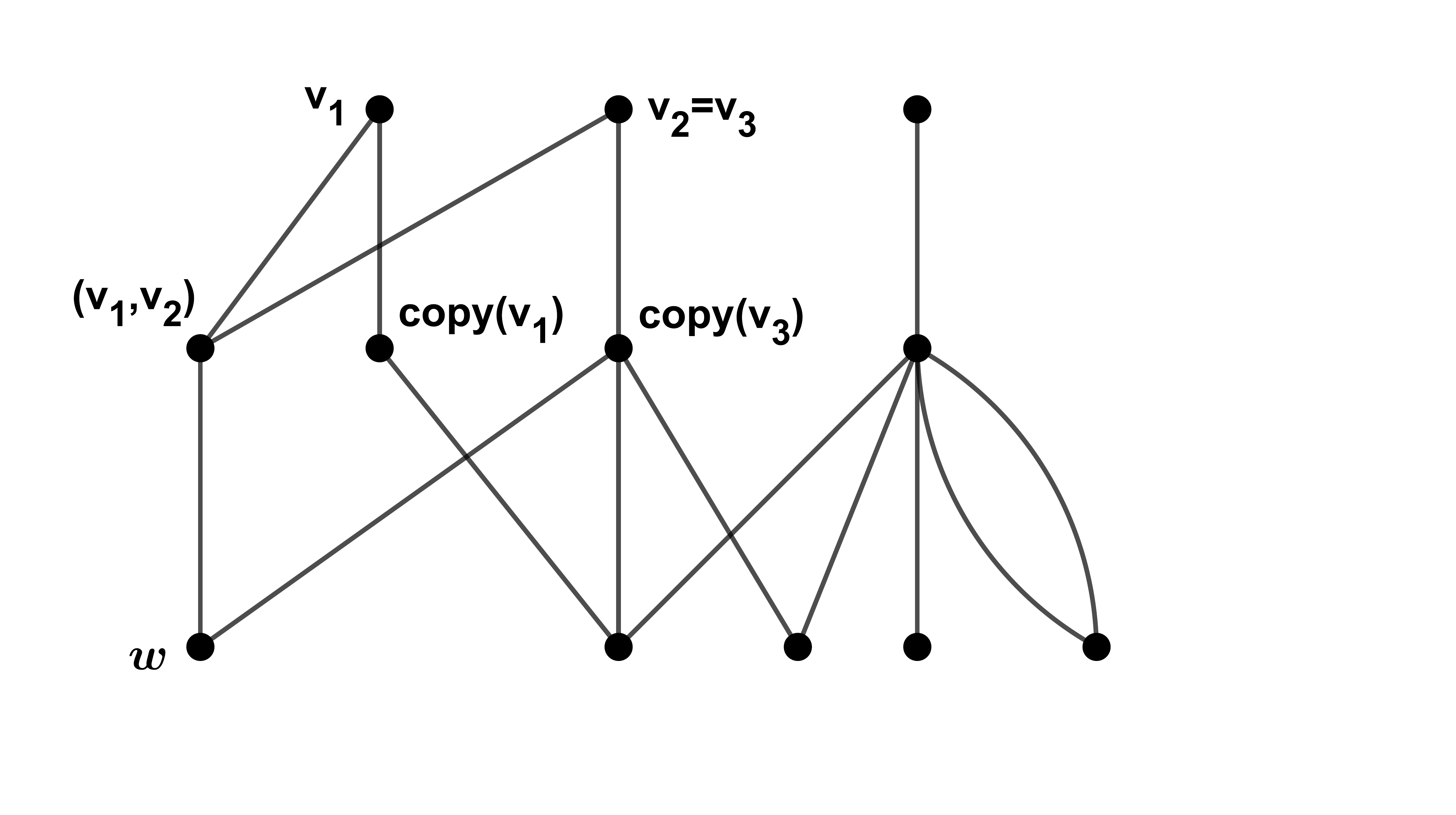}%
\end{center}
\vspace{-20pt}
\caption{Thinning a bipartite graph}%
\label{fig2}%
\end{figure}
\end{proof}

It is not difficult to see that  in Theorem~\ref{thAR}, given an aperiodic ergodic automorphism $T$, one can require that the constructed graded graph~$\Gamma$ satisfies the conditions of Proposition~\ref{prop3} (see~\cite{V82} and~\cite{VZ18}). This proposition allows one to embed the path space~$\T(\Gamma)$ of such a graph into~$\T(\uag)$, and thus we arrive at the following theorem.

\begin{theorem}[metric universality of the uniadic graph]\label{th1}
For every ergodic automorphism~$T$ of a~Lebesgue space there is a central measure~$\mu$ on the path space $\T(\uag)$ of the uniadic graph~$\uag$ such that the adic shift on the space $(\T(\uag),\mu)$ is isomorphic to~$T$.
\end{theorem}

\begin{corollary}\label{cor1}
For every ergodic automorphism~$T$ of a~Lebesgue space  $(X,\nu)$ there is a combinatorially definite basic filtration~$\Xi$ on~$(X,\nu)$.
\end{corollary}

\begin{proof}
Use Theorem~\ref{th1} to find a central measure $\mu$ on the path space $\T(\uag)$ of the uniadic graph~$\uag$ such that the adic shift on the space
 $(\T(\uag),\mu)$ is isomorphic to~$T$. The graph $\uag$ is minimal in the sense of Definition~\ref{mingraph}, hence, by Proposition~\ref{prop1}, its tail filtration is combinatorially definite for the adic shift isomorphic to~$T$.
\end{proof}

Theorem~\ref{th1} essentially states that for every aperiodic automorphism~$T$ of a Lebesgue space  $(X,\nu)$ there is a measurable mod~0 injective map~$f$ from $X$ to the path space $\T(\uag)$ of the uniadic graph that sends $T$ to the adic shift. If we fix the space $X$ and the transformation $T$ and vary the invariant measure~$\nu$, then the map $f$ described above will, in general, vary, since the construction of an adic realization of a given automorphism depends on the measure (see~\cite{V82}). But can one make the map $f$ independent of the measure~$\nu$?

Arguing in this way, we arrive at the following ``Borel'' question, which was raised in~\cite{VZ17}. Let $X$ be a standard Borel space and $T$ be an (aperiodic) Borel automorphism of~$X$. Does there exist a graded graph~$\Gamma$ endowed with an adic structure and a Borel-measurable embedding~$f$ of~$X$ into the path space~$\T(\Gamma)$ that sends $T$ to the adic shift?

It turns out that the answer to this question is positive. If $T$ is an aperiodic Borel automorphism of a separable metric space~$X$, then one can construct such a graph $\Gamma$ and such an embedding~$f$. Moreover, by Proposition~\ref{prop3}, the same uniadic graph~$\uag$ can serve as a Borel universal graph.

\begin{theorem}[Borel universality of the uniadic graph, \cite{VZ18}]\label{thborel}
Let $T$ be an aperiodic Borel automorphism of a separable metric space~$X$. Then there exists a Borel subset $\hat X \subset X$ such that $\mu(\hat X)=1$ for every $T$-invariant Borel measure on~$X$ and a Borel-measurable injective map~$f$ from $\hat X$ to the path space of the uniadic graph that sends $T$ to the adic shift.
\end{theorem}

A complete proof of this theorem will be published separately (see~\cite{VZ18}). Here we give only a sketch of the proof. It essentially follows the proof of Theorem~\ref{thAR}, except that all steps of the construction should now be Borel, i.e., independent of the measure.

The first step of the proof, as in the case of Theorem~\ref{thAR}, is a weakening of Rokhlin's lemma. A~Borel version of Rokhlin's lemma, unlike the classical one, seems far from being a trivial problem. For example, if we consider the shift $T$ on the space $\{0,1\}^\mathbb{Z}$, it says that for every  $\eps>0$ there is a~Borel subset~$B\subset\{0,1\}^\mathbb{Z}$ such that $B\cap TB=\emptyset$ and  for every $T$-invariant aperiodic measure $\mu$ on $\{0,1\}^\mathbb{Z}$, we have~$\mu(B\cup TB)>1-\eps$. The problem of finding a measure-free proof of the lemma, i.e., proving its Borel version, was posed by V.~A.~Rokhlin in a conversation with the first author.

A Borel version of Rokhlin's lemma was proved by B.~Weiss and E.~Glasner~\cite[p.~628, Proposition~7.9]{GV06}.

\begin{lemma}\label{lemGV}
Let $T$ be a homeomorphism of a Polish space $(X, \rho)$. Let $n \in \mathbb{N}$ and $\eps>0$. Then there exists a Borel subset $B\subset X$ such that the sets $B, TB, \dots, T^{n-1}B$ are pairwise disjoint and
$$
\mu(B\cup TB \cup \dots \cup T^{n-1}B)>1-\eps
$$
for every $T$-invariant aperiodic measure $\mu$ on $X$.
\end{lemma}

To prove Theorem~\ref{thborel}, as in the case of Theorem~\ref{thAR}, one should first weaken this version of the lemma by dropping the periodicity condition for approximating automorphisms, and then apply it  repeatedly, passing after each iteration to the derivative automorphism on the constructed set~$B$. In the language of filtrations, the proof consists in constructing a Borel basic filtration of the automorphism~$T$, coloring it, and embedding the colored filtration into the path space of the uniadic graph. For details, see~\cite{VZ18}.

\end{document}